\documentclass[12pt]{amsart}

\usepackage{amssymb}
\usepackage{amsmath}
\usepackage{amsthm,cite}

\newtheorem{theorem}{Theorem}[section]

\makeatletter \@addtoreset{equation}{section} \makeatother

\def\ddt{\frac{d}{dt}}

\numberwithin{equation}{section}

\newcommand\lam{\lambda}
\newcommand\na{\nabla}

\newcommand\Del{\Delta}

\newcommand\SL{\text{SL}(2, \mathbb{R})}
\def\ddt{\frac{d}{dt}}

\begin{document}

\bibliographystyle{plain}
\title[]
{Eigenvalues under the Ricci flow of model geometries}
\author{Songbo Hou}
\address{Department of Applied Mathematics, College of Science, China Agricultural
University,  Beijing, 100083, P.R. China}
\email{housb10@163.com}

\date{}
\def\baselinestretch{1}

\begin{abstract}

In this paper, we study the evolving behaviors of the first eigenvalue of the Laplace-Beltrami operator under the normalized Ricci flow of model geometries. In every Bianchi class,  we estimate the derivative of the eigenvalue. Then we construct monotonic quantities under the Ricci flow and obtain upper and lower bounds for the eigenvalue.
\end{abstract}
\maketitle \baselineskip 18pt
\section{Introduction}
The first eigenvalue of the Laplace-Beltrami operator plays an important role in studying the geometry and topology of manifolds. Especially the estimates of upper and lower bounds of the first eigenvalue have attracted many attentions and are an interesting subject.

 The classic results focus on the manifolds with fixed metrics. The research of the eigenvalue under the Ricci flow originates from \cite{PG02}. In that paper, Perelman proved that the first eigenvalue of $-\Del+R/4$ is nondecreasing under the Ricci flow, where $R$ is the scalar curvature of a evolving metric on a closed manifold. Using the monotonicity property, he derived that  there are no nontrivial steady or expanding breathers.
 Later, Cao \cite{Cao07} considered  eigenvalues of the operator $-\Del+\frac{R}{2}$ on manifolds with nonnegative curvature operator.  He derived the evolution equation and got that the eigenvalues are nondecreasing along the Ricci flow. In \cite{JFL07},  Li  improved Cao's result without assuming nonnegative curvature operator,  based on the same technique. The monotonicity of the first eigenvalues of $-\Del +aR\,(a\geq \frac{1}{4})$  was proved by  Cao using evolution formulae  \cite{Cao08} and by Li using entropy functionals \cite{JFL07}. Similar results hold under the normalized Ricci flow for the case $a=\frac{1}{4}$ with  nonpositive average scalar curvature \cite{Cao08}.  In \cite{CHL12},  we derived  monotonicity formulae for the first eigenvalue of $-\Del +aR \,(0<a\leq 1/2)$ under the Ricci flow and the normalized Ricci flow, and obtained various estimates on closed surfaces in terms of different Euler characteristic classes.
Especially the first eigenvalue multiplied by $e^{rt}$ is nondecreasing under the normalized Ricci flow if $R\geq 0$, where $r$  is the average of $R$.
Similar result holds  for $-\Del +aR \,(a\geq \frac{1}{4})$  without any
curvature assumption \cite{Cao08}.

For the Laplace-Beltrami operator, Ma \cite{Ma06} studied monotonicity of the first eigenvalue along  Ricci flow on a compact  domain  with Dirichlet boundary condition. Ling \cite{Ling07} proved the appropriate
multiples of the eigenvalues are monotonic along the normalized Ricci flow on closed manifolds.  In fact,  it is difficulty to get bounds and monotonicity formulae under the Ricci flow since it is hard to control  the scalar curvature and the Ricci tensor appearing in the evolution equation. The results in \cite{Ma06, Ling07} all need suitable assumptions on the scalar curvature or  Einstein tensor.  The results in  \cite {CHL12} show that  we can control the curvature and construct monotonic quantities because the estimate is reduced to an ODE problem on the Riemannian surface. For  an
orientable closed surface,  estimates of the eigenvalue of $p$-Laplace operator  were also obtained in \cite{WWZ10} by the ODE method. Related references include \cite{Hou,Ling2, Ling3}. A natural question is what is the behavior of the eigenvalue on 3-manifolds without more conditions on the curvature.  The study on homogeneous manifolds will provide us useful information about the general case.

 Since homogeneous geometries are the basic decomposed pieces
 of the geometrization of 3-manifolds,  it is significant to understand the behavior of the Ricci flows in this basic case. In \cite{IJ92}, J. Isenberg and M. Jackson studied the Ricci flows on locally  homogeneous 3-manifolds and described  their behaviors explicitly. Later, Knopf and McLeod \cite{KM01} studied the quasi-convergence of the Ricci flows of homogeneous models. Isenberg, Jackson and Lu
 \cite{IJL06} analyzed  the behaviors of the Ricci flow on locally homogenous closed 4-manifolds. Lott \cite{JL07} used the concept of groupoids and solitons to interpret the longtime behaviors of the Ricci flow on locally homogeneous closed manifolds further.

 The homogeneous geometries on simply connected 3-manifolds consist of nine classes and are divided into two sets. The first set contains $\mathbb{R}^3$, $\text{SU(2)}$, $\text{SL}(2,\mathbb{R})$, Heisenberg, $E(1,1)$ and
 $E(2)$. The other set contains $H(3)$, $H(2)\times \mathbb{R}^1$, and $\text{SO}(3)\times \mathbb{R}^1$.  The notations $H(n)$, $E(1,1)$ and $E(2)$ denote
 the group of isometries of the hyperbolic $n$-space, the group of isometries of the plane with flat Lorentz metric, the group of isometries of the Euclidian plane, respectively. The first set is called Bianchi classes \cite{IJ92}. For any given metric $g_0$ on one of Bianchi classes,
 there is  a Milnor frame $\{X_1,X_2,X_3\}$ so that the metric and the Ricci tensors are diagonal and such property is preserved by the Ricci flow.  Denote $\{\theta^1,\theta^2,\theta^3\}$ the dual coframe.  Then the Ricci flow takes the form\\
$$g(t)=A(t)(\theta^1)^2+B(t)(\theta^2)^2+C(t)(\theta^3)^2,$$ and  is reduced to an ODE system  involving  three variables $\{A(t), B(t), C(t)\}$.

   In this paper, we consider the first eigenvalue under the normalized Ricci flow on locally homogeneous closed 3-manifolds of Bianchi classes.     We  obtain  various  monotonicity formulae  and  upper and lower bounds of the eigenvalue under the Ricci flow. We wish our results could  be extended to the general manifolds. The key of our method is to  compare the components of the evolving Ricci tensor, and then control the derivative of the eigenvalue. Finally, the integration yields the results.

 \section {Evolution equation}
 In this section, we give the evolution equation for the first eigenvalue $\lam$. This is  a special case of Lemma 3.1 in \cite{CHL12}. For completeness, we provide the proof here.
 \begin{theorem}
 Let $(M, g(t))$  be a solution to the normalized Ricci flow  on a locally homogeneous  3-manifold. Denote $-\Delta$ the Laplace-Beltrami operator and $\lam(t)$ the first eigenvalue. Assume that
$u(x,t)>0$ satisfies
$$-\Delta u=\lam u,$$ with $\int u^2(x,t)d\mu=1.$
Then along the normalized Ricci flow, we obtain
\begin{align}\ddt\lam=\int (2R_{ij}\na_iu\na_ju)d\mu-\frac{2}{3}R\lam.\end{align}
\end{theorem}
 \begin{proof}
Following calculations in \cite {Cao07}, then we have
\begin{align*}
\ddt\lam=\int\left(\frac{2r}{3}u\Delta u -2uR_{ij}\na_i\na_ju\right)d\mu
\end{align*}
where $$r=\frac{\int_MRd\mu}{\int_Md\mu}$$  is the average of the scalar curvature, which equals  $R$ on locally homogeneous manifolds.
Integrating by parts and using the contracted Bianchi identity, we have
\begin{align*}
-\int 2uR_{ij}\na_i\na_jud\mu=\int \left[(2u\na_iR_{ij})\na_ju\right]d\mu+\int (2R_{ij}\na_iu\na_ju)d\mu,
\end{align*}
and
\begin{align*}
\int (2\na_iR_{ij})u\na_jud\mu &=\int u(\na_jR)\na_jud\mu=-\int Ru\Delta ud\mu-\int R|\na u|^2d\mu\\
&=\lam R-\lam R=0.
\end{align*}
We obtain
\begin{align*}
\ddt\lam=\int (2R_{ij}\na_iu\na_ju)d\mu-\frac{2}{3}R\lam.
\end{align*}
\end{proof}

\section{$\mathbb{R}^3$}
In this class,
$$g(t)=g_0, \,\forall  \,t\geq 0$$ and $\lam(t)$ is a constant.

 \section{\text{SU}(2)}

Given a metric $g_0$, we fix a Milnor frame $\{X_i\}_1^3$ such that
$$\left[X_2,X_3\right]=X_1,\,\,\,\left[X_3,X_1\right]=X_2,\,\,\,\left[X_1,X_2\right]=X_3.$$
 Under the normalization $A_0B_0C_0=1$, the nonzero components of Ricci tensor are (refer to  examples on page 171 in \cite{BLN})

\begin{equation}
\left\{
\begin{aligned}
&R_{11}=\frac{1}{2}A[A^{2}-(B-C)^{2}],\\
&R_{22}=\frac{1}{2}B[B^{2}-(A-C)^{2}],\\
&R_{33}=\frac{1}{2}C[C^{2}-(A-B)^{2}],
\end{aligned}\right.
\end{equation}
and the scalar curvature is
\begin{equation}R=\frac{1}{2}[A^{2}-(B-C)^{2}]+\frac{1}{2}[B^{2}-(A-C)^{2}]+\frac{1}{2}[C^{2}-(A-B)^{2}].\end{equation}
The Ricci flow equations are
\begin{equation}
\left\{
\begin{aligned}
&\frac{dA}{dt}=\frac{2}{3}A\left[-A(2A-B-C)+(B-C)^2\right],\\
&\frac{dB}{dt}=\frac{2}{3}B\left[-B(2B-A-C)+(A-C)^2\right],\\
&\frac{dC}{dt}=\frac{2}{3}C\left[-C(2C-A-B)+(A-B)^2\right].
\end{aligned} \right.
\end{equation}

 We assume without loss of generality that $A_{0}\geq B_{0}\geq C_{0}$.
\vskip 10pt

\noindent {\bf Lemma 4.1.} (Isenberg and Jackson \cite{IJ92} )
{\it  Let $(A,B,C)$ be a solution to system (4.3). Then we have the following results:

(1) $A\geq B\geq C$ and $A-C\leq (A_0-C_0)e^{-2C_0^2t}$.

(2) The flow converges exponentially to the fixed points $A=B=C=1$, namely the round metric on the three sphere.}

In the following, we denote $\tau$ a constant  which may vary from line to line and from section to section.
We get the following theorem.
\begin{theorem}
Let $\lam(t)$ be the  first eigenvalue of $-\Del$. Then there is time $\tau$ such that
$\lam(t) e^{\int_{\tau}^{t}(\frac{2}{3}R-2R_{33})dt}$ is nondecreasing along the normalized  Ricci flow, whereas $\lam(t) e^{\int_{\tau}^{t}(\frac{2}{3}R-2R_{11})dt}$ is nonincreasing.  Moreover, we get

$$e^{\frac{2(A_0-C_0)}{C_0^2}\left(e^{-2C_0^2t}-e^{-2C_0^2\tau}\right)}\lam(\tau)\leq \lam(t)\leq e^{\frac{5(C_0-A_0)}{2C_0^2}\left(e^{-2C_0^2t}-e^{-2C_0^2\tau}\right)}\lam(\tau)$$ for $t\geq \tau$.

As $t$ goes to $\infty$, $\lam(t)$ approaches a constant which corresponds to the eigenvalue of the round metric on the three sphere.
\end{theorem}
\begin{proof}

By Lemma 4.1, we have $B-C$, $A-C$ and $A-B$ all converge exponentially to zero.  Then $R_{ij}$ approaches $\frac{1}{2}$ and $R$ approaches $\frac{3}{2}$ as $t$ goes to infinity. Using (4.1) and Lemma 4.1, we arrive at
\begin{align*}
R_{11}-R_{22}&=\frac{1}{2}[A^3-A(B-C)^2]-\frac{1}{2}[B^{3}-B(A-C)^2]\\
&=\frac{1}{2}(A-B)[(A+B)^2-C^2]\\
&\geq 0
\end{align*}
and
\begin{align*}
R_{22}-R_{33}&=\frac{1}{2}[B^3-B(A-C)^2]-\frac{1}{2}[C^{3}-C(A-B)^2]\\
&=\frac{1}{2}(B-C)[(B+C)^2-A^2]\\
&\geq 0
\end{align*}
after a time $\tau$.
Thus we obtain
$$R_{11}\geq R_{22}\geq R_{33}$$
and
$$2R_{33}\lam-\frac{2}{3}R\lam\leq \frac{d\lam}{dt}\leq 2R_{11}\lam-\frac{2}{3}R\lam$$ if $t\geq \tau$.

Then
$\lam(t) e^{\int_{\tau}^{t}(\frac{2}{3}R-2R_{33})dt}$ is nondecreasing along the normalized Ricci flow, whereas $\lam(t) e^{\int_{\tau}^{t}(\frac{2}{3}R-2R_{11})dt}$ is nonincreasing.

 Since $2C-A-B\leq 0$, then $C$ is nondecreasing and consequently $C(t)\leq 1$, $A(t)\geq 1$. Furthermore, we get
\begin{align*}
&2R_{33}-\frac{2}{3}R\\
&=C[C^{2}-(A-B)^{2}]-\frac{1}{3}[A^2+B^2+C^2-(B-C)^2-(A-C)^2-(A-B)^2]\\
&\geq C[C^{2}-(A-B)^{2}]-\frac{1}{3}(A^2+B^2+C^2)\\
&\geq  C^3-C(A-C)^2-A^2\\
&\geq C^3-(A-C)^2-A^3\\
&=-(A-C)(A^2+AC+C^2+A-C)\\
&\geq -4(A-C)\geq -4(A_0-C_0)e^{-2C_0^2 t}
\end{align*}
and \begin{align*}
&2R_{11}-\frac{2}{3}R\\
&=A[A^{2}-(B-C)^{2}]-\frac{1}{3}[A^2+B^2+C^2-(B-C)^2-(A-C)^2-(A-B)^2]\\
&\leq  A^3-C^2+(A-C)\\
&\leq A^3-C^3+(A-C)\leq 5(A-C)=5(A_0-C_0)e^{-2C_0^2 t}
\end{align*}
after a time $\tau$.

Then $$ -4(A_0-C_0)e^{-2C_0^2 t}\leq \frac{1}{\lam}\frac{d\lam}{dt}\leq 5(A_0-C_0)e^{-2C_0^2 t}$$
which implies  $\lim_{t\rightarrow \infty}\lam(t)$ exists.

Integration for $\tau$ to $t$ yields
$$e^{\frac{2(A_0-C_0)}{C_0^2}\left(e^{-2C_0^2t}-e^{-2C_0^2\tau}\right)}\lam(\tau)\leq \lam(t)\leq e^{\frac{5(C_0-A_0)}{2C_0^2}\left(e^{-2C_0^2t}-e^{-2C_0^2\tau}\right)}\lam(\tau).$$
\end{proof}

\section{$\SL$  }
Given a metric $g_0$, we fix a Milnor frame such that
$$\left[X_2,X_3\right]=-X_1,\,\,\,\left[X_3,X_1\right]=X_2,\,\,\,\left[X_1,X_2\right]=X_3.$$
Under the normalization $A_0B_0C_0=1$, the nonzero curvature components are
\begin{equation}
\left\{
\begin{aligned}
&R_{11}=\frac{1}{2}A[A^{2}-(B-C)^{2}],\\
&R_{22}=\frac{1}{2}B[B^{2}-(A+C)^{2}],\\
&R_{33}=\frac{1}{2}C[C^{2}-(A+B)^{2}],\\
&R=\frac{1}{2}[A^{2}-(B-C)^{2}]+\frac{1}{2}[B^{2}-(A+C)^{2}]+\frac{1}{2}[C^{2}-(A+B)^{2}].\end{aligned}\right.
\end{equation}
Then the Ricci flow equations are
\begin{equation}
\left\{
\begin{aligned}
&\frac{dA}{dt}=\frac{2}{3}[-A^2(2A+B+C)+A(B-C)^2],\\
&\frac{dB}{dt}=\frac{2}{3}[-B^2(2B+A-C)+B(A+C)^2],\\
&\frac{dC}{dt}=\frac{2}{3}[-C^2(2C+A-B)+C(A+B)^2].
\end{aligned} \right.
\end{equation}

Without loss of generality, we may assume that $B_0>C_0$.
\vskip 10pt

\noindent {\bf Lemma 5.1.} {(Isenberg and Jackon  \cite{IJ92})} {\it Let $(A,B,C)$ be a solution to system (5.2). Then we have the following results:

(1) $B\geq C$ for all time $t$.

(2)  $B(t)\geq C_0+\frac{2}{3}t$, $C(t)\geq C_0+\frac{2}{3}t$  and $A\leq \left( C_0+\frac{2}{3}t\right)^{-2}$.

(3) There is time $\tau$ such that $A\leq B$ for all $t\geq \tau$,  $B-C\leq (B_{\tau}-C_{\tau})e^{-kt}$ and $B\leq B_{\tau}+\frac{4}{3}t$,  where $C_{\tau}:=C(\tau)$, $k:=\frac{10}{3}C_{\tau}^2.$

The Ricci flow evolves toward the pancake degeneracy.}
\begin{theorem}
Let $\lam(t)$ be the  first eigenvalue of $-\Del$. Then there is time $\tau$ such that
$\lam(t) e^{\int_{\tau}^{t}(\frac{2}{3}R-2R_{33})dt}$ is nondecreasing along the normalized Ricci flow, whereas $\lam(t) e^{\int_{\tau}^{t}(\frac{2}{3}R-2R_{11})dt}$ is nonincreasing. Moreover, we get
 $$ \lam(\tau)e^{-2(t-\tau)}\leq \lam(t)\leq \lam(\tau)\left(\frac{C_0+\frac{2}{3}t}{C_0+\frac{2}{3}\tau}\right)^3$$ for $t\geq \tau$.
\end{theorem}
\begin{proof}
First, let us compare $R_{11}$ with $R_{22}$:
\begin{align*}
R_{22}-R_{11}&=\frac{1}{2}[B^{3}-B(A^{2}+C^{2}+2AC)]-\frac{1}{2}[A^{3}-A(B^{2}+C^{2}-2BC)]\\
&=\frac{1}{2}[(B-A)(A^{2}+B^{2}+2AB-C^{2})-4ABC]\\
&=\frac{1}{2}[(B-A)(A+B+C)(A+B-C)-4]
\end{align*}
since $ABC=1$.

Then by Lemma 5.1, we have
\begin{align*}
&  (B-A)(A+B+C)(A+B-C)\\
 & \leq B(A+2B)(A+B-C)\\
& =A^{2}B+2AB^{2}+AB(B-C)+2B^{2}(B-C),
\end{align*}

$$A^{2}B\leq \left(C_{0}+\frac{2}{3} t\right)^{-4}\left(B_{\tau} +\frac{4}{3}t \right),$$
$$2AB^{2}=2ABC\times\frac{B}{C}=\frac{2B}{C},$$
$$AB(B-C)\leq \left(C_{0}+\frac{2}{3} t\right)^{-2}\left(B_{\tau} +\frac{4}{3}t \right)\left(B_{\tau}-C_{\tau}\right)e^{-kt},$$
and
$$2B^{2}(B-C)\leq 2\left(B_{\tau} +\frac{4}{3}t \right)^{2}\left(B_{\tau}-C_{\tau}\right)e^{-kt}$$
after a time $\tau$.

Noting that $B/C$ approaches  1 and other terms shrink to zero as $t$ goes  to $\infty$ , we get  $R_{11}>R_{22}$ with $t\geq \tau$.

Next we compare $R_{11}$ with $R_{33}$. By Lemma 5.1,  it is easy to see that $$A(t)\geq \frac{1}{(B_{\tau}+\frac{4}{3}t)^2}$$  and then $R_{11}>0$ and $R_{33}<0$ if $t$ is large enough.

Finally for $R_{22}$ and $R_{33}$, we get the following identity
\begin{align*}
R_{22}-R_{33}&=\frac{1}{2}B[B^{2}-(A+C)^{2}]-\frac{1}{2}C[C^{2}-(A+B)^{2}]\\
&=\frac{1}{2}(B-C)[(B+C)^{2}-A^{2}]
\end{align*}
and obtain $R_{22}\geq R_{33}$ after a time $\tau$.
So we have $R_{11}(t)>R_{22}(t)\geq R_{33}(t)$ if $t\geq \tau$.

Then we can rewrite the evolving equation for $\lam$
$$\left(2R_{33}-\frac{2}{3}R\right)\lam\leq \frac{d\lam}{dt}\leq \left(2R_{11}-\frac{2}{3}R\right)\lam.$$
This yields that
$$\frac{d}{dt}\left(\lam (t) e^{\int_{\tau}^{t}(\frac{2}{3}R-2R_{33})dt}\right)\geq 0$$
and
$$\frac{d}{dt}\left(\lam (t) e^{\int_{\tau}^{t}(\frac{2}{3}R-2R_{11})dt}\right)\leq 0.$$

Then
$\lam(t) e^{\int_{\tau}^{t}(\frac{2}{3}R-2R_{33})dt}$ is nondecreasing for $t\geq \tau$ along the normalized Ricci flow, whereas $\lam(t) e^{\int_{\tau}^{t}(\frac{2}{3}R-2R_{11})dt}$ is nonincreasing.

Moreover, the behaviors of $A,B,C$ imply the existence of $\tau$ such that
\begin{align*}
&2R_{33}-\frac{2}{3}R\\
&=C[C^{2}-(A+B)^{2}]+\frac{1}{3}(A^{2}+B^{2}+C^{2}+2AC+2AB-2BC)\\
&=\frac{1}{3}A^{2}+\frac{2}{3}AC+\frac{2}{3}AB+\frac{1}{3}(B-C)^{2}-C(B+C)(B-C)\\
&-2ABC-A^{2}C\\
&\geq \frac{1}{3}(A^{2}+AC+AB)-2-A^{2}C\\
&\geq -2
\end{align*}  after $\tau$ since $B-C$ decays exponentially to zero and $A\leq \left(C_0+\frac{2}{3}t \right)^{-2}$,

\begin{align*}
&2R_{11}-\frac{2}{3}R\\
&= A[A^{2}-(B-C)^{2}]+\frac{1}{3}(A^{2}+B^{2}+C^{2}+2AB+2AC-2BC)\\
&\leq \frac{1}{3}(A^{2}+(B-C)^{2}+2AB+2AC+3A^{3})\\
&\leq \frac{1}{3}(2A^{2}+2AB+2AC)\\
&\leq \frac{1}{3}(2AB+2AB+2AB)\\
&=2AB \leq \frac{2}{C}\leq \frac{2}{C_0+\frac{2}{3}t}
\end{align*}
 for all $t\geq \tau$.

Integration from $\tau$  to $t$  gives
 $$ \lam(\tau)e^{-2(t-\tau)}\leq \lam(t)\leq \lam(\tau)\left(\frac{C_0+\frac{2}{3}t}{C_0+\frac{2}{3}\tau}\right)^3.$$

 \end{proof}
\section{Heisenberg}
In this class, given a metric $g_0$, there is fixed Milnor frame such that
$$\left[X_2,X_3\right]=X_1,\,\,\,\left[X_3,X_1\right]=0,\,\,\,\left[X_1,X_2\right]=0.$$

 Under the normalization $A_0B_0C_0=1$, the curvature components for metrics are
\begin{equation}
\left\{
\begin{aligned}
&R_{11}=\frac{1}{2}A^{3},\\
&R_{22}=-\frac{1}{2}A^{2}B\\
&R_{33}=-\frac{1}{2}A^{2}C,\\
&R=-\frac{1}{2}A^{2}.
\end{aligned}\right.
\end{equation}
The  Ricci flow equations are then
\begin{equation}
\left\{
\begin{aligned}
&\ddt A=-\frac{4}{3}A^3,\\
&\ddt B=\frac{2}{3}A^2B,\\
&\ddt C=\frac{2}{3}A^2C.
\end{aligned}
\right.
\end{equation}
 and the solution   is
\begin{equation}
\left\{
\begin{aligned}
&A=A_{0}\left(1-\frac{16}{3}R_{0}t\right)^{-1/2},\\
&B=B_{0}\left(1-\frac{16}{3}R_{0}t\right)^{1/4},\\
&C=C_{0}\left(1-\frac{16}{3}R_{0}t\right)^{1/4}.
\end{aligned}\right.
\end{equation}
where $R_{0}=-\frac{1}{2}A_{0}^{2}$.

Then Ricci flow in this class approaches the pancake degeneracy.

Assume that $B_0\geq C_0$. We get the following theorem.
\begin{theorem}
Let $\lam(t)$ be the  first eigenvalue of $-\Del$. Then  $\lam(t) e^{\int_{0}^{t}(A^{2}B-\frac{1}{3}A^{2})dt}$ is nondecreasing  along the normalized Ricci flow, whereas $\lam(t) e^{\int_{0}^{t}-(\frac{1}{3}A^{2}+A^{3})dt}$ is nonincreasing. Moreover, we get
$$\lam(t)\geq \lam(0)\left[\left(1-\frac{16}{3}R_0 t\right)^{1/8}e^{-\frac{3B_0}{2}\left[(1+\frac{8}{3}A_0^2t)^{1/4}-1\right]}\right]$$
and
$$\lam(t)\leq \lam(0)\left[\left(1-\frac{16}{3}R_0 t\right)^{1/8}e^{-\frac{A_0}{4}\left[(1+\frac{8}{3}A_0^2t)^{-1/2}-1\right]}\right].$$
\end{theorem}

\begin{proof}
By (6.1) and (6.3), we have $R_{11}> R_{33}\geq R_{22}. $  It follows from (2.1) that
\begin{equation}\left(\frac{1}{3}A^{2}-A^{2}B\right)\lam\leq\frac{d}{dt}\lam\leq \left(\frac{1}{3}A^{2}+A^{3}\right)\lam.\end{equation}
Thus we get
$$\ddt\left(\lam(t) e^{\int_{0}^{t}(A^{2}B-\frac{1}{3}A^{2})dt}\right)\geq 0$$
and
$$\ddt\left(\lam(t) e^{\int_{0}^{t}-(\frac{1}{3}A^{2}+A^{3})dt}\right)\leq 0.$$

Then $\lam(t) e^{\int_{0}^{t}(A^{2}B-\frac{1}{3}A^{2})dt}$ is nondecreasing  along the normalized Ricci flow, whereas $\lam(t) e^{\int_{0}^{t}-(\frac{1}{3}A^{2}+A^{3})dt}$ is nonincreasing.

Furthermore, we have
$$\frac{1}{3}A^{2}-A^{2}B=\frac{A_0^2}{3\left(1-\frac{16}{3}R_{0}t\right)}-\frac{A_{0}^2B_{0}}{\left(1-\frac{16}{3}R_{0}t\right)^{3/4}}$$
and $$\frac{1}{3}(A^{2}+A^{3})=\frac{A_0^2}{3\left(1-\frac{16}{3}R_{0}t\right)}+\frac{A_0^3}{3\left(1-\frac{16}{3}R_{0}t\right)^{3/2}}.$$
Integration from 0 to $t$ yields
\begin{align*}\lam(t)&\geq \lam(0)\left[\left(1-\frac{16}{3}R_0 t\right)^{-\frac{A_0^2}{16R_0}}e^{\frac{3A_0^2B_0}{4R_0}\left[\left(1-\frac{16}{3}R_{0}t\right)^{1/4}-1\right]}\right]\\
&=\lam(0)\left[\left(1-\frac{16}{3}R_0 t\right)^{1/8}e^{-\frac{3B_0}{2}\left[\left(1+\frac{8}{3}A_0^2t\right)^{1/4}-1\right]}\right]
\end{align*}
and
\begin{align*}\lam(t)&\leq \lam(0)\left[\left(1-\frac{16}{3}R_0 t\right)^{-\frac{A_0^2}{16R_0}}e^{\frac{A_0^3}{8R_0}\left[\left(1-\frac{16}{3}R_0t\right)^{-1/2}-1\right]}\right]\\
&=\lam(0)\left[\left(1-\frac{16}{3}R_0 t\right)^{1/8}e^{-\frac{A_0}{4}\left[\left(1+\frac{8}{3}A_0^2t\right)^{-1/2}-1\right]}\right].
\end{align*}
\end{proof}
\section{E(1,1)}
Given a metric $g_0$, we choose a fixed Milnor frame such that
$$\left[X_1,X_2\right]=0,\,\,\,\left[X_2,X_3\right]=-X_1,\,\,\,\left[X_3,X_1\right]=X_2.$$
 Under the normalization $A_0B_0C_0=1$, the nonzero curvature components of the metric  are
\begin{equation}
\left\{
\begin{aligned}
&R_{11}=\frac{1}{2}A(A^2-B^2),\\
&R_{22}=\frac{1}{2}B(B^2-A^2),\\
&R_{33}=-\frac{1}{2}C(A+B)^2,\\
&R=-\frac{1}{2}(A+B)^{2}.\end{aligned}\right.
\end{equation}
The Ricci flow equations are
\begin{equation}
\left\{
\begin{aligned}
&\frac{dA}{dt}=\frac{2}{3}[-2A^3-AB(A-B)],\\
&\frac{dB}{dt}=\frac{2}{3}[-2B^3+AB(A-B)]\\
&\frac{dC}{dt}=\frac{2}{3}C(A+B)^2.
\end{aligned} \right.
\end{equation}
Without loss of generality, we may assume  that $A_0\geq B_0$.
\vskip 10pt

\noindent{\bf Lemma 7.1.} (Isenberg and Jackon \cite{IJ92}) {\it  Let $(A,B,C)$ be a solution to system (7.2). Then we have the following results:

(1) $$B_0(1+\frac{8}{3}A_0^2t)^{-1/2}\leq B\leq A\leq A_0(1+\frac{8}{3}A_0^2t)^{-1/2}$$ and
$$C_0+\frac{4}{3}t\leq C\leq C_0+\frac{8}{3}\left(\frac{A_0}{B_0}\right)t.$$

(2) $$(A_0-B_0)(1+\frac{8}{3}A_0^2t)^{-2}\leq A-B\leq (A_0-B_0)(1+\frac{8}{3}B_0^2t)^{-2}$$ and

 $$(A_0+B_0)(1+\frac{8}{3}A_0^2t)^{-1}\leq A+B\leq (A_0+B_0)(1+\frac{8}{3}B_0^2t)^{-1}.$$

 Then Ricci flow in this class approaches the cigar degeneracy.}
\begin{theorem}
Let $\lam(t)$ be the  first eigenvalue of $-\Del$. Then there is time $\tau$  such that $\lam(t) e^{\int_{\tau}^{t}-\left[\frac{1}{3}(A+B)^2-C(A+B)^2\right]dt}$ is nondecreasing  along the normalized Ricci flow,
whereas $\lam(t) e^{\int_{\tau}^{t}-\left[\frac{1}{3}(A+B)^2+A(A^2-B^2)\right]dt}$ is nonincreasing. Moreover, we get

$$\lam(\tau)(\frac{t}{\tau})^{-c_2}\leq \lam(t)\leq \lam(\tau)e^{c_1(\frac{1}{\tau}-\frac{1}{t})}$$ for $t\geq \tau$,
where $c_1$ and $c_2$ are two constants.
 \end{theorem}
\begin{proof}
By (7.1), we have
$$R_{22}-R_{33}=\frac{1}{2}(A+B)(B^2-AB+AC+BC).$$
By (1) in Lemma 7.1, we know there is time $\tau$ such that $CB-AB\geq 0$ for $t\geq \tau$. Then we get $$R_{11}\geq R_{22}\geq R_{33}$$ and
$$\left[\frac{1}{3}(A+B)^2-C(A+B)^2\right]\lam\leq\ddt \lam\leq \left[\frac{1}{3}(A+B)^2+A(A^2-B^2)\right]\lam$$ after $\tau$.

Then  $\lam(t) e^{\int_{\tau}^{t}-\left[\frac{1}{3}(A+B)^2-C(A+B)^2\right]dt}$ is nondecreasing  along the normalized Ricci flow,
 whereas $\lam(t) e^{\int_{\tau}^{t}-\left[\frac{1}{3}(A+B)^2+A(A^2-B^2)\right]dt}$ is nonincreasing.

Furthermore, by Lemma 7.1 we estimate
\begin{align*}
\frac{1}{3}(A+B)^2-C(A+B)^2& \geq -\frac{(A_0+B_0)^2\left(C_0+\frac{8}{3}\left(\frac{A_0}{B_0}\right)t\right)}{\left(1+\frac{8}{3}B_0^2 t\right)^2}\\
&\geq -\frac{c_2}{t}
\end{align*}
and
\begin{align*}
\frac{1}{3}(A+B)^2+A^2(A^2-B^2)&\leq \frac{1}{3} (A_0+B_0)\left(1+\frac{8}{3}B_0^2 t\right)^{-2}+(A+B)^2(A-B)\\
& \leq \frac{1}{3} (A_0+B_0)\left(1+\frac{8}{3}B_0^2 t\right)^{-2}\\
&+2A_0(A_0^2-B_0^2)\left(1+\frac{8}{3}B_0^2 t\right)^{-4}\\
&\leq c_1 t^{-2}
\end{align*}
if $t\geq \tau$, where $c_1$ and $c_2$ are two constants.

Integration from $\tau$ to $t$ yields
$$\lam(\tau)(\frac{t}{\tau})^{-c_2}\leq \lam(t)\leq \lam(\tau)e^{c_1(\frac{1}{\tau}-\frac{1}{t})}.$$
\end{proof}
\section{E(2)}
Given a metric $g_0$, we choose a  Milnor frame such that
$$\left[X_2,X_3\right]=X_1,\,\,\,\left[X_3,X_1\right]=X_2,\,\,\,\left[X_1,X_2\right]=0.$$
Under the normalization $A_0B_0C_0=1$, the nonzero curvature components are

\begin{equation}
\left\{
\begin{aligned}
&R_{11}=\frac{1}{2}A(A^2-B^2),\\
&R_{22}=\frac{1}{2}B(B^2-A^2),\\
&R_{33}=-\frac{1}{2}C(A-B)^2,\\
&R=-\frac{1}{2}(A-B)^{2}.
\end{aligned}\right.
\end{equation}
Then the Ricci flow equations are
\begin{equation}
\left\{
\begin{aligned}
&\frac{dA}{dt}=-\frac{2}{3}A(2A+B)(A-B),\\
&\frac{dB}{dt}=\frac{2}{3}B(2B+A)(A-B),\\
&\frac{dC}{dt}=\frac{2}{3}C(A-B)^2.
\end{aligned} \right.
\end{equation}

Without loss of generality, we may assume  that $A_0\geq B_0$.
\vskip 10pt
\noindent{\bf Lemma 8.1.} ( Isenberg and Jackon \cite{IJ92} )  {\it Let $(A,B,C)$ be a solution to system (8.2). Then we have the following results:

(1) $A\geq B$, $B_0\leq A\leq A_0$, $B_0\leq B\leq A_0$  and $C_0\leq C\leq C_0\frac{A_0}{B_0}$.

(2) $A-B\leq (A_0-B_0)e^{-4B_0^2t}$.

 The flow converges to a flat metric.}
\begin{theorem}
Let $\lam(t)$ be the  first eigenvalue of $-\Del$. Then there is time $\tau$ such that  $\lam(t) e^{\int_{\tau}^{t}-\left[\frac{1}{3}(A-B)^2+B(B^2-A^2)\right]dt}$ is nondecreasing  along the normalized Ricci flow,
whereas $\lam(t) e^{\int_{\tau}^{t}-\left[\frac{1}{3}(A-B)^2+A(A^2-B^2)\right]dt}$ is nonincreasing. Moreover, we get

$$e^{\frac{A_0^2(A_0-B_0)}{2B_0^2}( e^{-4B_0^2t}-e^{-4B_0^2\tau})}\lam(\tau)\leq \lam(t)\leq e^{\frac{A_0^2(B_0-A_0)}{4B_0^2}( e^{-4B_0^2t}-e^{-4B_0^2\tau})}\lam(\tau)$$ for $t\geq \tau$.

As $t$ goes to $\infty$, $\lam(t)$ approaches a constant which corresponds to the eigenvalue of a flat metric.
 \end{theorem}
\begin{proof}
It follows from (8.1) and Lemma 8.1  that
$$R_{22}-R_{33}=\frac{1}{2}(A-B)\left[C(A-B)-B(A+B)\right]\leq 0 $$ after a time $\tau$.
Then we have $R_{11}\geq R_{33}\geq R_{22}$ and
$$\left[\frac{1}{3}(A-B)^2+B(B^2-A^2)\right]\lam\leq \ddt \lam\leq \left[\frac{1}{3}(A-B)^2+A(A^2-B^2)\right]\lam$$ if $t\geq \tau.$

Then  $\lam(t) e^{\int_{\tau}^{t}-\left[\frac{1}{3}(A-B)^2+B(B^2-A^2)\right]dt}$ is nondecreasing  along the normalized Ricci flow,
 whereas $\lam(t) e^{\int_{\tau}^{t}-\left[\frac{1}{3}(A-B)^2+A(A^2-B^2)\right]dt}$ is nonincreasing.

More precisely,  we get
$$
\frac{1}{3}(A-B)^2+B(B^2-A^2)\geq -2A_0^2(A_0-B_0)e^{-4B_0^2t},
$$

$$\frac{1}{3}(A-B)^2+A(A^2-B^2)=\frac{1}{3}(A-B)(A-B+A^2+AB)\leq A _0^2(A_0-B_0)e^{-4B_0^2t}\lam,$$
and $$-2A_0^2(A_0-B_0)e^{-4B_0^2t}\lam\leq \ddt \lam(t)\leq A _0^2(A_0-B_0)e^{-4B_0^2t}\lam$$
if $t\geq \tau.$

Then $$-2A_0^2(A_0-B_0)e^{-4B_0^2t}\leq \frac{1}{\lam}\ddt \lam(t)\leq A _0^2(A_0-B_0)e^{-4B_0^2t}$$  which implies  $\lim_{t\rightarrow \infty}\lam(t)$ exists.

Integration from $\tau$ to $t$ gives
$$e^{\frac{A_0^2(A_0-B_0)}{2B_0^2}( e^{-4B_0^2t}-e^{-4B_0^2\tau})}\lam(\tau)\leq \lam(t)\leq e^{\frac{A_0^2(B_0-A_0)}{4B_0^2}( e^{-4B_0^2t}-e^{-4B_0^2\tau})}\lam(\tau)$$
for $t\geq \tau$.
\end{proof}
\vskip 30pt
{\bf Acknowledgement}

The author's research was supported by National Natural Science Foundation of China (Grant No.11001268) and Chinese University Scientific Fund (2014QJ002). The author would like to thank Professor Xiaodong Cao and Professor Laurent Saloff-Coste for their suggestions and interests in this work.

 \end{document}